\documentclass[12pt]{article}

\usepackage{amssymb}
\usepackage{amsmath}
\usepackage{amscd}
\usepackage{amsthm}
\usepackage{graphicx}

\usepackage{color}          
\usepackage{epsfig}
\usepackage{authblk}

\newcommand{\C} {\mathbb C}

\newcommand{\h} {\hat}
\newcommand{\mc}{\mathcal}

\newtheorem{theorem}{Theorem}

\newtheorem{corollary}[theorem]{Corollary}
\newtheorem{proposition}[theorem]{Proposition}
\newtheorem*{definition}{Definition}
\newtheorem*{Example}{Example}
\title{On decomposable rational maps.}

\author{Carlos Cabrera  and Peter Makienko\\
Instituto de Matem\'aticas,\\ Unidad Cuernavaca. UNAM}

\begin{document}

\maketitle
\footnotetext{This work was partially supported by PAPIIT project IN 100409.}
\begin{abstract}
If $R$ is a rational map, the Main Result is a uniformization Theorem for
the space of decompositions of the iterates of $R$. Secondly, we show that
Fatou conjecture holds for decomposable rational maps. 

\end{abstract}

\section{Introduction}

This paper gives a dynamical approach to the algebraic problem of
decomposition of rational maps. That is to describe the set of decompositions
of a rational map $R$, along with the decompositions of all its iterates $R^n$.
We want to link geometric structures with the decomposition of rational maps.
To this end, we construct a space which describes the space of decomposition of
the cyclic semigroup generated by $R$.

We found that the fact that a map is decomposable impose dynamical
consequences. In particular, we show using elementary arguments that the
Fatou conjecture is true for decomposable rational maps. 

We would like to thank M. Zieve for useful comments and discussions and for
kindly providing his example of a prime rational map which is virtually
indecomposable.

\section{On stability of decomposable maps}
Given a rational map $R$, the Julia set $J(R)$ is the smallest completely
invariant closed set in the Riemann sphere
$\bar{\C}$, with at least $3$
points. The Fatou set $F(R)$ is the complement of the Julia set on $\bar{C}$.
By definition, the set $F(R)$ is open and completely invariant.

A map $R$ is decomposable if there are maps $R_1$ and $R_2$, of degree
at least $2$, such that $R=R_1\circ R_2$. In this section, we study stability
properties for decomposable rational maps. The simple fact that the
maps $$R=R_1\circ R_2\textnormal{ and }\tilde{R}=R_2\circ R_1$$ are
semiconjugated, provides
arguments to show that $J$-stability implies hyperbolicity for decomposable
maps.  The Fatou conjecture, as restated in \cite{MSS}, states that all
$J$-stable maps are hyperbolic.

First, we recall the definitions of $J$-stability, more details can be found in
\cite{MSS} and \cite{McMSull}. Let $(X,d_1)$ and $(Y,d_2)$ be metric spaces, a
map
$\phi:X\rightarrow Y$ is called $K$-quasiconformal, in Pesin's sense if, for
every $x_0\in X$
$$\limsup_{r\rightarrow 0}\left\{
\frac{\sup\{|\phi(x_0)-\phi(x_1)|:|x_0-x_1|<r\}}{\inf\{
|\phi(x_0)-\phi(x_1)|:|x_0-x_1|<r\}}\right\}\leq K.$$

Let us recall that two rational maps $R_1$ and $R_2$ are
 $J$\textit{-equivalent}, if there is a homeomorphism
$h:J(R_1)\rightarrow J(R_2)$, which is quasiconformal in Pesin's
sense and conjugates $R_1$ to $R_2$.

Given a family of maps $\{R_w\}$ depending holomorphically on a
parameter $w\in W$, a map $R_{w_0}$ in $\{R_w\}$ is called
$J$\textit{-stable} if, there is a neighborhood $V$ of $w_0$ such
that, $R_w$ is $J$-equivalent to $R_{w_0}$ for all $w\in V,$ and the
conjugating homeomorphisms depend holomorphically on $w$.

\begin{theorem}\label{thm.decom}
Let $R=R_1\circ R_2$ and $\tilde{R}=R_2\circ R_1$, such that $deg(R_i)>1$ for
$i=1,2$. If both maps, $R$ and $\tilde{R}$, are $J$-stable. Then $R$ and
$\tilde{R}$ are hyperbolic.
\end{theorem}

Let $Cr(R)$ denote the set of critical points of $R$. It is well known (see for
instance \cite{L}),
that a $J$-stable map $R$ is hyperbolic if and only if the set $Cr(R)$ belongs
to $F(R)$.

\begin{proof}

Since $R$ is $J$-stable, then $R$ is in general position with respect to the
Julia set, that is, $R$ has no critical relations on $J(R)$. In particular, the
local degree of each critical point of $R$ is $2$. To prove the claim,
we  will show that there are no critical points in $J(R)$. First notice
that as a consequence of the Chain Rule, we have the equation 
$$Cr(R)=Cr(R_1\circ R_2)=R_2^{-1}(Cr(R_1))\cup Cr(R_2).$$

Let $x$ be a point in $R_2^{-1}(Cr(R_1))\cap J(R)$, since $J(R)$ is completely
invariant under $R$,  every point in $R_2^{-1}(R_2(x))$ belongs to $J(R)$ and
is a critical point of $R$. Also, because  there are no critical relations in
$J(R)$, the set $R_2^{-1}(R_2(x))$ consists of only one point. However $R_2$ has
degree at least $2$, hence $x$ is a critical point of $R_2$, but $R_2(x)$ is a
critical point of $R_1$, which implies that the local degree of $R$ in $x$ is at
least $4$. This contradicts the fact that there are no critical relations in
$J(R)$. Then $R_2^{-1}(Cr(R_1))$ belongs to the Fatou set $F(R)$.

There are two semiconjugacies between $R$ and
$\tilde{R}$ as shown in the following diagram:
$$\begin{CD} \C @>R>> \C\\
@V{R_2}VV @VV{R_2}V\\
\C @>\tilde{R}>> \C\\
@V{R_1}VV @VV{R_1}V\\
\C @>{R}>> \C\end{CD} $$
The first semiconjugacy, in fact $R_2$, sends $F(R)$ to $F(\tilde{R})$, hence
$Cr(R_1)=R_2(R_2^{-1}(Cr(R_1)))$ belongs to $F(\tilde{R})$. By
the same argument $Cr(R_2)$ is a subset of $F(R)$. Altogether 
$Cr(R)$ belongs to $F(R)$. Therefore the map $R$ is hyperbolic and, by the
symmetry of the argument, the map $\tilde{R}$ is also hyperbolic.

\end{proof}

The previous theorem has the following corollary which already was noted in
\cite{MakFarEast}.

\begin{corollary}\label{cor.hyp} The following statements are equivalent
\begin{itemize}
 \item The map $R$ is hyperbolic.
\item There exist $n>1$ such that $R^n$ is $J$-stable in $Rat_{d^n}$.
\item For every $n\geq 1$, the map $R^n$ is $J$-stable in $Rat_{d^n}$.
\end{itemize}
\end{corollary}

Intuitively, one could understand  the facts above considering the
dimension of the space of invariant line fields on the Julia set $J(R)$.  This
dimension is comparable with the number of critical values of $R^n$ on $J(R)$,
which grows linearly with respect to iteration. On the other hand, the condition
of $J$-stability in $Rat_{d^n}(\C)$ requires exponential growth of the
dimensions with respect to iteration, but these dimensions  should be comparable
with the number of critical points of $R^n$ on $J(R)$. The incompatibility of
the rate of these growths is a contradiction to the existence of invariant line
fields on $J(R)$.

Now let us consider for a given rational map $R$ the Hurwitz space $H(R)$,
that is, the set of
all rational maps with the same combinatorics for the first iteration,
equivalently $$H(R)=\{Q\in Rat(\C): \exists \phi \textnormal{ and } \psi \in
Homeo(\bar{\C}) \textnormal{ such that } Q\circ \phi=\psi \circ R\}.$$ 
When $R$ is a rational map in general position, with   $\deg(R)=d$, the space 
$H(R)$ is open and dense in the space $Rat_d(\C)$. 
Note that $H(R^n)$ consists of compositions of the form $R_1\circ R_2\circ ...
\circ R_n$ with $R_i\in H(R)$. It is not clear if $H(R^n)$ consists of all
compositions of this form. However, $$\dim(H(R^n))\leq n \dim H(R).$$
Now we can ask the analog of Corollary \ref{cor.hyp} for Hurwitz spaces:
\begin{enumerate}
 \item Assume that $R^n$ is $J$-stable in $H(R^n)$, is it true that $R^k$ is
$J$-stable at $H(R^k)$ for $k\neq n$? We expect an affirmative answer for
$k<n$.
\item Assume that $R^n$ is $J$-stable in $H(R^n)$ for all $n$, is it true that
$R$ is hyperbolic?
\end{enumerate}

The second question is actually a modified version of Fatou conjecture. These
questions make sense for entire and meromorphic transcendental maps.

The conditions of Theorem \ref{thm.decom} are too strong, it is enough that 
one of the maps, say $R=R_1\circ R_2$ is $J$-stable.

\begin{proposition}\label{prop.sym}  Let $R=R_1\circ R_2$ and
$\tilde{R}=R_2\circ R_1$, such that $deg(R_i)>1$ for $i=1,2$. If $R$ is
$J$-stable then $\tilde{R}$ is $J$-stable.
\end{proposition}

We will just sketch the proof of Proposition \ref{prop.sym}. Let us denote by
$QC_J(R)$, the $J$-\textit{stability component} of $R$. This is
the path connected component of the $J$-equivalence class of $R$ containing $R$.
We need the following theorem which was proved in \cite{McMSull}, see also
\cite{MSS}.

 \begin{theorem}[McMullen-Sullivan]\label{thm.McMSull} On every analytic family
$H$, the set of $J$-stable maps is open and dense. Moreover, the set of
structurally stable maps is also dense in $H$. 
 \end{theorem}

Let $H$ be an analytic family, by Theorem \ref{thm.McMSull}  the set
$U=QC_J(R)\cap H$ is an open set. Every holomorphically embedded disk  $D$ 
in $U$ containing $R$, depending on a complex parameter $t$, is equivalent to a
family of Beltrami coefficients $\mu_t$, whose associated quasiconformal maps
$f_t$ conjugate $R$ to $R_t$ along $D$. The maps $f_t$ form a holomorphic motion
of $J(R)$, using S\l{}odkowski's Extended $\lambda$-Lemma, the maps $f_t$ can be
extended to a neighborhood of $J(R)$ for every $t$. Moreover, the extension can
be taken to preserve the dynamics (see \cite[Theorem 1.7]{Slodkequiv}). Now
consider the push-forward operator $(R_2)_*$ which sends the family $\mu_t$ to
the family of Beltrami differentials $(R_2)_* \mu_t$ defined on a neighborhood
$W$ of $\tilde{R}$. The complementary components of $W$ can be taken to be
simply connected. With this choice, we can extend the maps $(R_2)_* \mu_t$ to
the whole sphere by attaching, with surgery, Blaschke maps on each complementary
component.
Solving the Beltrami equation for the resulting Beltrami coefficients will
induce a
family rational maps, $J$-equivalent to $\tilde{R}$. Thus $\tilde{R}$ is
$J$-stable.

The heart of the proof lies on the fact that the semiconjugacy defines an
bijective operator in the space of invariant line fields on the
Julia set. 

\section{A semigroup associated to a rational map $R$.}

In this section, for every rational map $R$ we construct a suitable semigroup
$S_R$, such that the space of analytic equivalences of $S_R$ uniformizes
the space of virtual decompositions of $R$. The semigroup $S_R$ will be a
semigroup of correspondences on the affine part $\mc{A}_R$ of $R$, as defined by
M. Lyubich and Y. Minsky in \cite{LM}.

Let us recall first  Lyubich and Minsky's construction, given  a rational
map $R$ defined in the Riemann sphere $\bar{\C}$, consider the inverse limit
$$\mc{N}_R=\{\h{z}=(z_1,z_2,...): R(z_{n+1})=z_n\}.$$ 

The natural extension of $R$ is the map $\h{R}:\mc{N}_R\rightarrow \mc{N}_R$
given in coordinates by $\h{R}(\h{z})_n=R(z_n)$. There is a family of maps 
$\pi_n:\mc{N}_R\rightarrow \bar{\C}$, the coordinate
projections, defined by $\pi_n(\h{z})=z_n$, which semiconjugates the action of
$\h{R}$ with $R$, that is $\pi_n\circ \h{R}=R\circ \pi_n$. By endowing
$\mc{N}_R$ with the topology of the restriction of Tychonoff topology, one
can show that the map $\h{R}$ is a
homeomorphism. The regular part $\mc{R}_R$ is
the maximal subset of $\mc{N}_R$ which admits a Riemannian structure, of complex
dimension one, compatible with the coordinate projections $\pi_n$. A leaf is a
path connected component of the regular part. A theorem by Lyubich and Minsky
(see \cite[Lemma 3.3]{LM}) states that, besides leaves associated to Herman
rings, all leaves are simply connected. The affine part $\mc{A}_R$ consists of
the regular points whose leaves are conformally isomorphic to complex plane
$\C$. Let $\mc{C}=\{ a_1,a_2,...,a_n\}$ be a repelling periodic cycle for $R$.
An invariant lift of $\mc{C}$ is the set of points $\h{a}$ in $\mc{N}_R$ such
that all the coordinates of $\h{a}$ belong to $\mc{C}$. Invariant lifts of
periodic repelling points belong to the affine part. Moreover,  the uniformazing
function of the leaves containing these invariant leaves is a Poincar\'{e}
function associated to $\mc{C}$. Since there are infinitely many of repelling
periodic cycles, the affine part consist of an infinite number of leaves.  

Let us remind that a holomorphic correspondence $K$ is a subset of a
product of complex spaces $B\times C$ such that, $K$ is the union of
countably many analytic varieties and, the projections are holomorphic and the
projection of $K$ to the first coordinate is surjective.

In a fiber $F$, of the form $\pi_n^{-1}(x)$ for $x\in\bar{\C}$, we can define
the set of deck transformations, or dual monodromies, which are given by the
correspondences $\pi_n^{-1}\circ \pi_n$. The fact that the
conformal structure on leaves is compatible with the projections $\pi_n$, means
that the leaf admits a conformal structure such that, in this structure, the
maps are holomorphic. In particular, given two leaf saturated sets $\mc{B}$ and
$\mc{C}$  in $\mc{A}_R$. If the cardinality of leaves in $\mc{B}$ is
at most countable, then $(\pi_n|_{\mc{B}})^{-1}\circ \pi_n|_\mc{C}$ is a
holomorphic correspondence in $\mc{B}\times \mc{C}$, here $\pi_n|_D$ denotes the
restriction of the map $\pi_n$ to the set $D$. When $\mc{B}=\mc{B}$, then
$(\pi_n|_{\mc{B}})^{-1}\circ \pi_n|_\mc{B}$ is a semigroup. 

If $L$ and $L'$ are two leaves in $\mc{A}_R$, then the map $\h{R}$ sends 
$(\pi_n|_{L'})^{-1}\circ \pi_n|_L$ to $(\pi_n|_{R(L')})^{-1}\circ
\pi_n|_{R(L)}$. In particular if $\mc{B}$ is a $\h{R}$-invariant, leaf
saturated, set in $\mc{A}_R$, the action of $\h{R}$ commutes with the action
of $(\pi_n|_{\mc{B}})^{-1}\circ \pi_n|_\mc{B}$ in $\mc{A}_R$.

A leaf $L$ in $\mc{A}_R$ is called periodic if there exist some $n$ such that 
$L$ is invariant under $\h{R}^n$. Let $\mc{C}(R)$ the set of all periodic leaves
in $\mc{A}_R$. The semigroup of deck correspondences is holomorphic
correspondence $(\pi_n|_{\mc{C}(R)})^{-1}\circ \pi_n|_{\mc{C}(R)}$ and will be
denoted by $\pi_n^{-1}\circ \pi_n$.

Now, let us define the semigroup
$S_R$ as the semigroup $\langle \mc{C}(R), \h{R},\pi_n^{-1}\circ
\pi_n\rangle$, generated by the constants maps on the set $\mc{C}(R)$, the
dynamics of $\h{R}$ and  $\pi_n^{-1}\circ
\pi_n$. We refer to $\mc{C}(R)$ as the
set of constants of $S_R$ and the dynamical part of $S_R$ will be the
semigroup generated by $\h{R}$ and $\{\pi_n^{-1}\circ \pi_n\}$.
  
\begin{definition}
A marked monomorphism $\rho:S_R\rightarrow S_{R_1}$ is a monomorphism that sends
constants to constants, maps analytically leaves to leaves, and
sends the dynamical part of $S_R$ to the dynamical part of $S_{R_1}$. That is,
the map $\rho$
sends the semigroup generated by $\h{R}$ to the semigroup generated by
$\h{R}_1$ and, the  action of deck transformations to the action of deck
transformations. 
\end{definition}

By definition, a marked monomorphism also sends fibers, of the family of
projections $\pi_n$, to fibers. An analytical isomorphism is a marked
monomorphism $\rho$ whose inverse is also a marked monomorphism.

\begin{theorem}\label{th.equivsem}
If $S_1$ and $S_2$ are semigroups associated to $R_1$ and $R_2$,  and let
$\psi:S_1\rightarrow S_2$ be a marked monomorphism of semigroups, then up to
M\"obius conjugacy of the maps $R_1$ and $R_2$, there exist $\Psi:\C \rightarrow
\C$ such that the following diagram
commutes 
$$\begin{CD} S_1 @>\psi>> S_2\\
@V{\pi}VV @VV{\pi}V\\
C @>\Psi>> C\\
\end{CD} $$
Where $C$ and $C'$ denote either the sphere $\bar{\C}$, the plane $\C$, or the
puncture plane $\C^*$ whenever the exceptional sets of $R_1$, $R_2$ have $0$,
$1$ or $2$ points respectively. Moreover, the map $\Psi$ is M\"obius, if and
only if $\psi$ is an analytic isomorphism.

\end{theorem}

\begin{proof}
Since $\psi$ conjugates the action of the deck transformations $\pi_n^{-1}\circ
\pi_n$, it sends fibers of $\pi$ on fibers of $\pi$, hence induces a map $\Psi$
defined on the image of the projections $\pi_n$. The Riemannian structure on
$\mc{A}_R$ is consistent with the projections $\pi_n$, because $\psi$ preserves
the leaf structure on $\mc{C}(R)$ the map $\Psi$ is also analytic. If $\psi$ is
an isomorphism, the map $\psi$ has an inverse which descends to an analytic
inverse of $\Psi$, hence the map $\Psi$ is M\"obius.
\end{proof}

 Let us remind, see for example \cite{Ritt}, two
decompositions $R_1\circ R_2\circ...\circ R_m$ and $P_1\circ P_2\circ ...\circ
P_m$, are called equivalent if there are M\"{o}bius transformations $\gamma_i$,
for $i=1,...,m-1$, such that 
$$P_1=R_1\circ \gamma_1,\, P_i=\gamma_{i-1}^{-1}\circ R_i \circ \gamma_i,
\textnormal{ for } 1<i<m, \textnormal{ and } P_m=\gamma_{m-1}^{-1}\circ R_m.$$

\begin{definition} A rational map $R$, is called \textnormal{prime}, 
\textnormal{indecomposable}, if whenever we have $R=P\circ Q$, where $P$ and $Q$
are rational maps,  then either $P$ or $Q$ belong to $PSL(2,\C)$. A
decomposition of $R=R_1\circ R_2 \circ ... \circ R_n$ is called a
\textnormal{prime decomposition} if, and only if, each $R_i$ is prime of degree
at least $2$ for all $i$.

The rational map $R$ is called virtually decomposable if there exist a number
$n>0$ and prime rational maps $R_1$,...$R_m$ such that $R_1\circ R_2\circ
...\circ R_m$ is a decomposition of $R^n$ non equivalent to $R^n$. 
\end{definition}

Every decomposable rational map is virtually decomposable. As a consequence of
Ritt's theorems for polynomials, every virtually decomposable polynomial is
decomposable. Surprisingly for rational maps this statement is false, M. Zieve
constructed the following counterexample:

\begin{Example}
 Let $R(z)=\frac{(z-1)^2}{(z+1)^2}$, then one can check that 
$$R^2(z)=\frac{4z}{(z+1)^2} \circ z^2.$$
Which is non equivalent to $R\circ R$. It is remarkable that this
example appears already on rational maps of degree 2.
\end{Example}

Let $R$ be a virtually decomposable rational map, such that $R^n$ has other 
decompositions. Hence $R^{kn}$ has also other decompositions for every $k$. Is
it true that eventually there are no new decompositions? In other words, whether
the list of decompositions for the iterates $R^m$ is finitely generated, let
$g(R)$ be the number of generators of this list.

Let $\alpha(R)$ the supremum of the numbers $m$ for which new decompositions of
$R^m$ appear. In general, the number $\alpha(R)$ can be infinite. Clearly the
finiteness of $g(R)$ implies the finiteness of  $\alpha(R)$. It is
a problem to determine the reciprocal also holds, this is equivalent to the non
existence of infinitely many decompositions for a given rational map. 

In \cite{Zieve.Mueller}, P. M\"{u}ller and M. Zieve proved that if $P$ is a
polynomial of degree $d\geq 2$, and not associated to parabolic orbifolds
(definition below), then $\alpha(P)$ is bounded by $\log_2 d$. Moreover, this
bound is sharp, and for some exceptions both $\alpha(R)$ and $g(R)$ can be
infinite.

Michael Zieve suggested the conjecture that with exception of rational maps
associated to parabolic orbifolds, the number $\alpha(R)$ is bounded in terms
of the degree. We believe that, with the same exceptions, the number $G(R)$ is
bounded and is comparable with $\alpha(R) \dim(H(R)/Aff(\C))$, where $Aff(\C)$
acts on $H(R)$ by conjugation.

Let us remind that a parabolic orbifold is a Thurston orbifold $\mc{O}$ with
non negative  Euler characteristic. When the map $R$ is postcritically
finite, the only maps associated to parabolic orbifolds are maps that
are M\"obius conjugated to maps of the form 

\begin{equation} z\mapsto z^n, \textnormal{  Chebychev polynomials and Latt\`es
maps.}
\end{equation}

In the paper \cite{Rittper}, J. Ritt gave the description of all the solutions
of the equation of the form $$R_1\circ R_2=R_2\circ R_1.$$ In \cite{EreFunc},
A. Eremenko reformulated Ritt's theorem in dynamical terms. 

\begin{theorem}[Eremenko]\label{tm.Eremenko} Let $R_1$ and $R_2$ be a pair of
commuting rational maps.  Then either there exist a pair of numbers, $n$ and
$m$, such that $R_1^n=R_2^m$ or, there is a parabolic orbifold $\mc{O}$ such
that
maps $R_1$ and $R_2$ are covering maps from $\mc{O}$ to $\mc{O}$.
\end{theorem}

Maps $R$ associated to parabolic orbifolds have affine laminations $\mc{A}_R$
with special geometry as it is shown in the following theorem due to Lyubich
and Kaimanovich (see \cite{KaiLyu}). 

\begin{theorem}[Kaimanovich-Lyubich]\label{th.kailyub} The affine lamination
$\mc{A}_R$ admits a continuously varying Euclidean structure on leaves if and
only if the map admits a parabolic orbifold.
\end{theorem}

As a consequence of Theorem \ref{th.kailyub} we have the following
proposition.

\begin{proposition} If a map $R$ admits a parabolic orbifold then the
semigroup generated by the restrictions of deck $\langle
\pi_n^{-1}\circ \pi_n\rangle$ to leaves is a group of mappings.\end{proposition}

\begin{proof}
By Theorem \ref{th.kailyub}, the  leaves admit an Euclidean structure
compatible with
projections, this implies that, under a suitable uniformization for all leaves
in $\mc{C}(R)$,
the deck transformations act on leaves as a group of translations.  
\end{proof}

Let us define now the space of analytic equivalences  $A(S_R)$.

\begin{definition}
The space of analytic deformations of $S_R$ is the space of 
triples $(S_{R_1}, \rho_1,\rho_2)$, where $R\neq R_1$ and $\rho_1:S_R\rightarrow
S_{R_1}$ and $\rho_2:S_{R_1}\rightarrow S_R$ are marked monomorphisms.

We say that $(S_{R_2},\rho_1,\rho_2)$ and $(S_{R_3},\phi_1,\phi_2)$ are
analytically equivalent if and only if  there is an isomorphism
$\gamma$ from
$S_{R_2}$ and $S_{R_3}$. Let $A(S_R)$ denote the space of analytic equivalences
of $S_R$.
\end{definition}

If $(S_{R_1}, h, g)$ belongs to $A(S_R)$ then $h\circ g$ and $g\circ h$ commutes
with $\h{R}$ and $\h{R}_1$ respectively. Next theorem shows the correspondence
of the
space $A(S_R)$ with the number of
virtual decomposition of the iterates of $R$.

\begin{theorem}
The map $R$ is virtually decomposable if and only if $$card(A(S_{R^n}))>1.$$
 Moreover, the number of virtual
decompositions of $R$ is in one-to-one correspondence with the points of 
$A(S_R)$.

\end{theorem}

\begin{proof}
Assume that $R$ is virtually decomposable, then there exist $n$ such that 
$R^n$ has a decomposition $Q_1\circ Q_2$ such that $Q_i$ is not equivalent to
$R^j$ for some $j\leq n$. Then the semigroup associated to $Q_2\circ Q_1$ is
analytically equivalent to $S_R$ but not M\"obius equivalent, therefore
$card(A(S_{R^n}))>1$.  Now let us assume that there exist a number $n>0$, such
that  there is more than one
analytic equivalence for $S_{R^n}$. Then there are analytic equivalences $q_1$ 
and $q_2$ between $S_{R^n}$ and a semigroup $S_Q$, associated to a rational map
$Q$. By Theorem ~\ref{th.equivsem}, the map $q_1\circ q_2$ descends to an
analytic map $Q_1\circ Q_2$, defined on the Riemann sphere with at most 2
punctures,  hence $Q_1\circ Q_2$ is a rational map, such that $Q_1\circ Q_2$
commutes with $R$. By the same reasoning $Q_2\circ Q_1$ commutes with $Q$. Let
us assume that $R$ and $Q_1\circ Q_2$ are associated to a parabolic orbifold
$\mc{O}$, then  the semiconjugacies 

$$Q_1\circ R=Q\circ Q_1$$
$$Q_2\circ Q=R\circ Q_2,$$ 
imply that $Q$ is also associated to the same
parabolic orbifold $\mc{O}$. Hence
$S_Q=S_R$ which contradicts the definition of analytic deformation of $S_R$.

Then by  Theorem \ref{tm.Eremenko} the maps $R$ and $Q_1\circ Q_2$ have a common
iterate. That is, there are numbers $n_1$ and $n_2$ such that $R^{n_1}=(Q_1\circ
Q_2)^{n_2}$, by construction the map $Q_1\circ Q_2$ is not M\"obius equivalent
to $R^i$, hence the map $R$ is virtually decomposable.

\end{proof}      

\subsection{Decomposition graphs.} 

Let $\mc{D}$ be the semigroup of decomposable rational maps. We construct a
directed graph
$\mc{G}$ associated to $\mc{D}$, where the vertices are the elements of
$\mc{D}$ and there is a directed edge, from $R$ to $\tilde{R}$,
if there are two rational maps $R_1$ and $R_2$ such that $R=R_1\circ R_2$ and
$\tilde{R}=R_2\circ R_1$.  Given a map $R\in \mc{D}$, let $G(R)$ be the
connected component of $\mc{G}$ containing $R$. We call $G(R)$, the
\textit{graph based at} $R$.

\begin{Example}
Consider a map that has a decomposition $R=R_1\circ R_2 \circ R_3$. Then 
the graph based on $R$ contains, at least, a triangle, with vertices $R$,
$R_2\circ R_3 \circ R_1$ and $R_3\circ R_1\circ R_2$. Other decorations may
appear from other decompositions of the map $R$ as it is shown next.

Let us remind the first Ritt theorem for decomposition of polynomials.

\begin{theorem}[First Ritt's Theorem]\label{th.Rittfirst}
Let $P_1\circ ... \circ P_m$ and $Q_1\circ ... \circ Q_n$ be two primes
decompositions of a polynomial $P$, then $n=m$.
\end{theorem}

In \cite{Berg2}, an erratum of the paper \cite{Berg}, W. Bergweiler wrote the
following counterexample, which is due to M. Zieve,  to the First Ritt's
theorem for rational maps.

$$R(z):=z^3\circ {\frac {{z}^{2}-4}{z-1}}\circ {\frac {{z}^{2}+2}{z+1}}={\frac
{z
\left( z-8 \right) ^{3}}{ \left( z+1 \right) ^{3}}}\circ z^3.$$

One can check that each factor is prime. The graph $G(R)$
contains the triangle above together with a segment, based on $R$, connecting
$R$ with ${\frac {z \left( z-8 \right) ^{3}}{ \left( z+1 \right) ^{3}}}\circ
z^3.$

\end{Example}

The graphs $G(R)$ give a topological realization of the
decomposition structure of $R$. That is, two maps $R_1$ and $R_2$ have the
same decomposable set if, and only if, the graphs $G(R_1)$ and $G(R_2)$ are
isomorphic.  However the graphs $G(R)$, as defined so far, are 
very big. For every $\gamma \in PSL(2,\C)$, we have $R=(R_1\circ
\gamma^{-1})\circ ( \gamma \circ R_2)$. Hence in the graph  based on $R$, the
point $R$ is connected to the maps $\gamma\circ \tilde{R}\circ \gamma^{-1}$.
To refine the information in $G(R)$ we consider a quotient of $\mc{D}$ by the
conjugacy action of $PSL(2,\C)$. Under this quotient the  graphs $G(R)$ become
finite, and makes sense to consider their fundamental groups $\pi_1(G(R),R)$.
This groups provide invariants for the decomposition
structure of rational maps

It is possible to simplify even more the information in $G(R)$ by considering a
CW completion of the graph. Namely, complete every  triangle induced by
$R_1\circ R_2\circ R_3$ by a $3$-simplex, to every tetrahedron induced by
$R_1\circ R_2 \circ R_3 \circ R_4$ by a $4$-simplex and, so on. In this setting,
the cohomology groups of this CW-complex give other set of invariants.
 
We finally note that given a map $R$, there is a correspondence between the
vertices in $G(R)$ and the elements in $A(R)$, and such that the edges in
$G(R)$ correspond to marked monomorphisms. 

\bibliographystyle{amsplain}
\bibliography{workbib}

\end{document}